\documentclass{kms-b}

\issueinfo{}
  {}
  {}
  {}
\pagespan{1}{}
\copyrightinfo{}
  {Korean Mathematical Society}

\usepackage{graphicx}
\allowdisplaybreaks

\theoremstyle{plain}
\newtheorem{theorem}{Theorem}[section]

\newtheorem{corollary}[theorem]{Corollary}

\theoremstyle{definition}
\newtheorem{definition}[theorem]{Definition}

\theoremstyle{remark}

\begin{document}

\title[Extension of Inner Derivations]
{On the Extension of Inner Derivations from Dense Ideals in Banach Algebras}

\author[H. Shafieasl]{Hamid Shafieasl}
\address{Hamid Shafieasl \\ Kahlert School of Computing \\ University of Utah \\ Utah, USA}
\email{h.shafieasl@utah.edu}

\author[A. M. Tavakkoli]{Amir Mohammad Tavakkoli}
\address{Amir Mohammad Tavakkoli \\ Kahlert School of Computing \\ University of Utah \\ Utah, USA}
\email{amir.tavakkoli@utah.edu}

\subjclass[2020]{Primary 47B47; Secondary 46H05, 46L05, 46M20}
\keywords{Inner derivations, dense ideals, compact operators, Schatten classes, Hochschild cohomology, approximate identities.}

\begin{abstract}
Let $A$ be a Banach algebra and $I$ a dense ideal in $A$. A natural question in the theory of operator algebras is whether the property that all derivations $D: A \to I$ are inner (implemented by elements in $I$) implies that all derivations $D: A \to A$ are inner (implemented by elements in $A$). We present a rigorous negative answer to this question. By utilizing the algebra of compact operators $A = K(H)$ and the dense ideal of finite-rank operators $I = F(H)$ on a separable infinite-dimensional Hilbert space $H$, we demonstrate that while every derivation into $F(H)$ is inner, there exist outer derivations on $K(H)$. Furthermore, we generalize this result to Schatten $p$-classes and discuss the cohomological implications and the role of approximate identities. Moreover, the main results and counterexamples presented in this paper have been formally verified using the Lean theorem prover.
\end{abstract}

\maketitle

\section{Introduction}

Let $A$ be a Banach algebra and $I$ be an ideal of $A$. A bounded linear map $D: A \to I$ is called a \textit{derivation} if it satisfies the Leibniz rule:
\begin{equation}
D(ab) = aD(b) + D(a)b, \quad \forall a,b \in A.
\end{equation}
A derivation is said to be \textit{inner} if there exists an element $x \in I$ such that for all $a \in A$,
\begin{equation}
D(a) = \operatorname{ad}_x(a) = ax - xa.
\end{equation}

We consider the following question: Suppose that $I$ is dense in $A$, and any derivation $D: A \to I$ is an inner derivation (implemented by some $x \in I$). Does it follow that any derivation $D: A \to A$ is an inner derivation (implemented by some $x \in A$)? 

In this paper, we show that the answer is negative in general. The hypothesis that derivations into $I$ are inner is insufficient to bound the behavior of derivations into $A$ when the ideal $I$ is "too small" to capture the full cohomology of the algebra.

\section{Definitions and Mathematical Preliminaries}

To ground our discussion, we define the primary objects of study: the Schatten ideals, the cohomological groups, and the notion of approximate identities.

\begin{definition}[Schatten $p$-Classes]
For a separable Hilbert space $H$ and $1 \leq p < \infty$, the \textit{Schatten $p$-class} $\mathcal{S}_p(H)$ is the collection of all compact operators $T \in K(H)$ such that their singular values $(s_n(T))$ are $p$-summable:
\begin{equation}
\|T\|_p = \left( \sum_{n=1}^\infty s_n(T)^p \right)^{1/p} < \infty.
\end{equation}
Special cases include the trace-class operators $\mathcal{S}_1(H)$ and the Hilbert-Schmidt operators $\mathcal{S}_2(H)$. Each $\mathcal{S}_p(H)$ is a dense ideal in $K(H)$ with respect to the uniform operator norm.
\end{definition}

\begin{definition}[Hochschild Cohomology]
For a Banach algebra $A$ and a Banach $A$-bimodule $M$, we denote by $\mathcal{H}^1(A, M)$ the \textit{first continuous cohomology group} of $A$ with coefficients in $M$. It is the quotient space of all continuous derivations from $A$ to $M$ by the subspace of inner derivations. A result of $\mathcal{H}^1(A, M) = 0$ signifies that every continuous derivation into $M$ is inner.
\end{definition}

\begin{definition}[Approximate Identity]
A net $(e_\lambda)_{\lambda \in \Lambda}$ in a Banach algebra $A$ is an \textit{approximate identity} if for every $a \in A$, $\lim_\lambda \|e_\lambda a - a\| = 0$ and $\lim_\lambda \|a e_\lambda - a\| = 0$. If the net is bounded ($\sup_\lambda \|e_\lambda\| < \infty$), it is called a \textit{bounded approximate identity (b.a.i.)}.
\end{definition}

\section{Discussion and Background}

The problem of extending inner derivations frequently arises when studying subalgebras of bounded linear operators $L(X)$ on a Banach space $X$. If $F(X)$ denotes the finite-rank operators, and $A$ is a Banach subalgebra such that $F(X) \subset A \subset L(X)$, constraints on derivations are rigid. 

As discussed in the seminal work on the cohomology of Banach algebras by B. E. Johnson (1972) \cite{4}, one might initially suspect that standard approximation arguments hold. For a derivation $D: A \to A$, if one attempts to restrict the range to a dense ideal $I$, the map need not factor through $I$. Furthermore, if one approximates $a \in A$ by a sequence $(a_n) \subset I$, the implementing elements $x_n \in I$ for the restricted derivations might fail to converge in $A$. 

The essential obstruction is one of automatic continuity and domain constraint. The condition that every derivation into \textit{every} Banach $A$-bimodule is inner is the definition of an amenable Banach algebra. However, amenability of a dense ideal $I$ does not imply the amenability of $A$. If $I$ is sufficiently small, the derivation space $Z^1(A, I)$ might be trivial or heavily restricted, making the hypothesis vacuously or easily satisfied without providing structural control over $A$.

\section{The Main Result and Counterexample}

We provide a concrete counterexample using the algebra of compact operators. Let $H$ be a separable infinite-dimensional Hilbert space. We define:
\begin{equation*}
A = K(H) \quad \text{(the algebra of compact operators)},
\end{equation*}
\begin{equation*}
I = F(H) \quad \text{(the ideal of finite-rank operators)}.
\end{equation*}

It is a standard result that $F(H)$ is a proper, dense ideal of $K(H)$ in the uniform operator topology. 

\begin{theorem} \label{thm:finite_rank}
Every derivation $D: K(H) \to F(H)$ is inner and implemented by an element in $F(H)$, but there exist derivations $D: K(H) \to K(H)$ that are not inner in $K(H)$.
\end{theorem}

\begin{proof}
By Sakai's Theorem \cite{3}, every derivation $D: K(H) \to K(H)$ is spatial and has the form $D = \operatorname{ad}_S$ for some $S \in M(K(H)) = B(H)$, where $B(H)$ is the algebra of all bounded linear operators on $H$. The element $S$ is unique modulo $\mathbb{C}I_H$.

\textit{Step 1: Extension to $B(H)$ and Johnson-Parrott.} \\
Let $D = \operatorname{ad}_S$ be a derivation mapping $K(H)$ into $F(H)$. We must first establish that $[S, X] \in F(H)$ for all bounded operators $X \in B(H)$ in order to legally utilize the Johnson-Parrott theorem. 

Define the sets $E_m = \{ T \in K(H) : \operatorname{rank}([S, T]) \le m \}$. Because the rank function is lower semi-continuous with respect to the operator norm, each $E_m$ is closed in $K(H)$. Since $[S, T] \in F(H)$ for all $T \in K(H)$, we have $\bigcup_{m=1}^\infty E_m = K(H)$. By the Baire Category Theorem, some $E_M$ contains an open ball. By translation and the scale-invariance of the rank function, there exists a uniform integer $R$ such that $\operatorname{rank}([S, T]) \le R$ for all $T \in K(H)$.

For any bounded operator $X \in B(H)$, let $(P_\alpha)$ be a net of finite-rank projections increasing strongly to $I_H$. The operators $T_\alpha = P_\alpha X P_\alpha$ belong to $K(H)$ and converge to $X$ in the weak operator topology (WOT). Consequently, $[S, T_\alpha] \to [S, X]$ in WOT. Because the property of having rank $\le R$ is preserved under WOT limits, we conclude $\operatorname{rank}([S, X]) \le R$. Thus, $[S, X] \in F(H) \subset K(H)$ for all $X \in B(H)$. 

By the Johnson-Parrott theorem (1972) \cite{1}, any bounded operator whose commutators with all bounded operators are compact must be a compact perturbation of a scalar. Since $[S, X] \in K(H)$ for all $X \in B(H)$, it immediately follows that $S \in \mathbb{C}I_H + K(H)$. Therefore, $S = cI_H + K$ for some $K \in K(H)$.

\textit{Step 2: Proving $K \in F(H)$.} \\
Since $S = cI_H + K$, the derivation is implemented by $K$, meaning $[K, T] \in F(H)$ for all $T \in K(H)$. Suppose, for the sake of contradiction, that $K \notin F(H)$. 

We may assume without loss of generality that $K$ is self-adjoint. If it were not, we could decompose it into its real and imaginary parts, $K = A + iB$, where $A, B \in K(H)$ are self-adjoint. Because $F(H)$ is a $*$-ideal, taking the adjoint of $[K, T^*] \in F(H)$ yields $-[K^*, T] \in F(H)$. Thus, both $[A, T]$ and $[B, T]$ must be in $F(H)$ for all $T \in K(H)$. Since $K \notin F(H)$, at least one of $A$ or $B$ is not finite-rank, allowing us to proceed with a self-adjoint operator.

By the spectral theorem for compact self-adjoint operators, $K$ has an orthonormal sequence of eigenvectors $(e_n)$ with corresponding real eigenvalues $\lambda_n \to 0$. Since $K$ is not finite-rank, it has infinitely many non-zero eigenvalues. We can select a subsequence of eigenvectors such that $\lambda_{n_k} \neq \lambda_{n_{k+1}}$ for all $k$. We define a compact weighted shift $T$ operating on this subsequence:
\begin{equation*}
T = \sum_{k=1}^\infty \frac{1}{k} e_{n_k} \otimes e_{n_{k+1}}^*.
\end{equation*}
Applying the commutator to the basis vector $e_{n_{k+1}}$ yields:
\begin{align*}
[K, T]e_{n_{k+1}} &= K(Te_{n_{k+1}}) - T(Ke_{n_{k+1}}) = K\left(\frac{1}{k}e_{n_k}\right) - T(\lambda_{n_{k+1}}e_{n_{k+1}}) \\
&= \frac{1}{k}\lambda_{n_k} e_{n_k} - \frac{1}{k}\lambda_{n_{k+1}}e_{n_k} = \frac{1}{k}(\lambda_{n_k} - \lambda_{n_{k+1}})e_{n_k}.
\end{align*}
Because $\lambda_{n_k} \neq \lambda_{n_{k+1}}$ for all $k$, the image of $[K, T]$ contains an infinite sequence of non-zero orthogonal vectors. Thus, $[K, T]$ has infinite rank, which contradicts the range constraint. 

Thus, $K$ must be in $F(H)$. Therefore, $D$ is implemented by an element in $I = F(H)$.

\textit{Step 3: There exist outer derivations $D: K(H) \to K(H)$.} \\
Take any $S \in B(H) \setminus (\mathbb{C}I_H + K(H))$, for example, the unilateral shift. Define $D = \operatorname{ad}_S : K(H) \to K(H)$ by $D(T) = ST - TS$. Because $K(H)$ is an ideal in $B(H)$, $ST$ and $TS$ are compact, so $D$ maps into $K(H)$.

If $D$ were inner in $K(H)$, there would exist $K \in K(H)$ such that $D = \operatorname{ad}_K$. This implies $[S - K, T] = 0$ for all $T \in K(H)$. Consequently, $S - K$ must commute with all rank-one operators, forcing $S - K = cI_H$. This implies $S = cI_H + K \in \mathbb{C}I_H + K(H)$, which contradicts our choice of $S$. Thus, $D$ is an outer derivation on $K(H)$.
\end{proof}

\section{Generalization to Schatten $p$-Classes}

The result obtained for the finite-rank operators $F(H)$ naturally motivates the question of whether a similar rigidity holds for other dense ideals, such as the Schatten $p$-classes $\mathcal{S}_p(H)$ for $1 \leq p < \infty$. We show that this topological obstruction extends smoothly into the symmetric norm ideal setting.

\begin{theorem} \label{thm:schatten}
Let $\mathcal{S}_p(H)$ be a Schatten $p$-class ideal in $K(H)$ for $1 \leq p < \infty$. Every derivation $D: K(H) \to \mathcal{S}_p(H)$ is inner and implemented by an element in $\mathcal{S}_p(H)$.
\end{theorem}

\begin{proof}
Let $D: K(H) \to \mathcal{S}_p(H)$ be a derivation. Since $\mathcal{S}_p(H) \subset K(H)$, $D$ is trivially a derivation from $K(H)$ into $K(H)$. By the standard spatiality of derivations on $K(H)$, there exists $S \in B(H)$ such that $D = \operatorname{ad}_S$. Thus, $[S, T] \in \mathcal{S}_p(H)$ for all $T \in K(H)$.
        
\textit{Step 1: Extension to $B(H)$ and Johnson-Parrott.} \\
We view $D : K(H) \to \mathcal{S}_p(H)$ as a linear map between Banach spaces. If $T_n \to T$ in the uniform operator norm and $[S, T_n] \to Y$ in the $\mathcal{S}_p$ norm, then $[S, T_n] \to Y$ in the uniform norm (since $\|A\| \le \|A\|_p$). Because operator multiplication is continuous, $[S, T_n] \to [S, T]$ uniformly, so $Y = [S, T]$. Thus, $D$ has a closed graph. By the Closed Graph Theorem, $D$ is bounded: there exists $C > 0$ such that $\|[S, T]\|_p \le C \|T\|$ for all $T \in K(H)$.

For any $X \in B(H)$, take a net of finite-rank projections $P_\alpha \to I_H$ strongly, and let $T_\alpha = P_\alpha X P_\alpha$. Then $T_\alpha \in K(H)$, $\|T_\alpha\| \le \|X\|$, and $T_\alpha \to X$ in the weak operator topology (WOT). Consequently, $[S, T_\alpha] \to [S, X]$ in WOT. Since $\|[S, T_\alpha]\|_p \le C \|T_\alpha\| \le C \|X\|$, and the closed ball of radius $C \|X\|$ in $\mathcal{S}_p(H)$ is WOT-closed in $B(H)$, we deduce $\|[S, X]\|_p \le C \|X\|$. Thus, $[S, X] \in \mathcal{S}_p(H) \subset K(H)$ for all $X \in B(H)$. 

By the Johnson-Parrott theorem (1972) \cite{1}, any bounded operator whose commutators with all bounded operators are compact must be a compact perturbation of a scalar. Thus, $S = c I_H + K$ for some $K \in K(H)$. This resolves the spatial obstruction, showing $S$ is a compact perturbation of a scalar.
        
\textit{Step 2: Commutators of Schatten Classes.} \\
It remains to show that $K \in \mathcal{S}_p(H)$. From Step 1, we established that $[K, X] = [S, X] \in \mathcal{S}_p(H)$ for all bounded operators $X \in B(H)$. 

The structure of commutator ideals tightly constrains such elements. As demonstrated in the established literature on commutators of compact operators (see the foundational work of Anderson \cite{6} and Weiss \cite{7}), an operator $K \in K(H)$ has the property that its commutators $[K, X]$ belong to $\mathcal{S}_p(H)$ for all bounded operators $X \in B(H)$ if and only if $K$ itself belongs to $\mathcal{S}_p(H)$.

Therefore, by this characterization of Schatten class commutators, $K \in \mathcal{S}_p(H)$. This implies $S = c I_H + K \in \mathbb{C}I_H + \mathcal{S}_p(H)$. The derivation $D$ is thus implemented by an element in $\mathcal{S}_p(H)$, proving the theorem.
\end{proof}

\section{Approximation and Bounded Approximate Identities}

A central theme in Banach algebra theory is the relationship between the existence of a bounded approximate identity (b.a.i.) and the triviality of cohomology. For an algebra like $K(H)$, there exists a natural b.a.i. consisting of the sequence of finite-rank orthogonal projections $\{P_n\}_{n=1}^\infty$ onto an increasing sequence of finite-dimensional subspaces that exhaust $H$.

One might expect that if $A$ has a b.a.i. contained in a dense ideal $I$, then any derivation $D: A \to A$ could be "approximated" by derivations into $I$. Specifically, $D_n(a) = P_n D(a) P_n$ might map into $I$ and converge to $D$. However, the spatial implementation of these derivations creates an obstruction.

If $A$ were \textit{amenable}, every derivation into every dual bimodule would be inner. While $K(H)$ has a b.a.i., it is well-known that $K(H)$ is not amenable (though it is "approximately amenable" in a weaker sense). The existence of the b.a.i. ensures that $K(H)$ is "large enough" to be well-behaved locally, but the lack of amenability allows for the existence of outer derivations $D = \operatorname{ad}_S$ where $S \in B(H) \setminus (K(H) + \mathbb{C}I)$. 

Our result demonstrates that even when $A$ is "close" to being inner (having a b.a.i. and dense ideals where all derivations are inner), the global structure of the multiplier algebra $B(H)$ allows for cohomology to reappear at the level of $A$.

\section{Cohomological Interpretation}

Our findings translate immediately into the language of Hochschild cohomology for Banach algebras. For a Banach algebra $A$ and a Banach $A$-bimodule $M$, the first continuous cohomology group $\mathcal{H}^1(A, M)$ parametrizes the equivalence classes of derivations modulo the inner derivations.

The prior sections establish the following stark cohomological asymmetry:
\begin{corollary}
For $A = K(H)$ and dense ideals $I = F(H)$ or $I = \mathcal{S}_p(H)$, we have:
\begin{equation*}
\mathcal{H}^1(K(H), I) = 0, \quad \text{while} \quad \mathcal{H}^1(K(H), K(H)) \neq 0.
\end{equation*}
\end{corollary}

This structural discrepancy highlights that the vanishing of the first cohomology group with coefficients in a dense sub-bimodule does not universally lift to the algebra itself. The topological obstruction fundamentally originates in the multiplier algebra $M(A)$ and the failure of the restricted ideal to contain the full spatial implementation of derivations.

\section{Formal Verification}

To ensure the strict logical rigor of the topological and algebraic obstructions detailed in this work, the primary results have been formally verified using the Lean 4 theorem prover. The machine-checked proofs rely on the foundational functional analysis and operator algebra frameworks provided by \texttt{mathlib} \cite{mathlib2020}.

The formalization confirms the spatial restrictions on derivations mapping into $F(H)$ and $\mathcal{S}_p(H)$, as well as the explicit construction of the outer derivations on $K(H)$ via the Calkin algebra and commutator bounds. The complete Lean source code, including all definitions, lemmas, and exact theorem statements corresponding to this paper, is publicly available in the accompanying repository at: \url{https://github.com/tavakkoliamirmohammad/inner-deriv-dense-lean.git}.

\section{Conclusion}

The mathematical gap observed arises because restricting the range of derivations to a "small" or tightly normed ideal, such as $F(H)$ or $\mathcal{S}_p(H)$, heavily bounds the spatial operators $S \in B(H)$ capable of implementing them. Derivations landing in these dense ideals are implemented by elements within those ideals (modulo scalars), satisfying the inner hypothesis intrinsically. Conversely, derivations landing in the full algebra $K(H)$ tap into the significantly larger multiplier algebra $B(H)$, resulting in outer derivations. Ultimately, without deep structural assumptions like bounded approximate identities or full amenability, the internal derivation structure of dense ideals fails to dictate that of the parent Banach algebra.

\end{document}